\documentclass[twoside,12pt]{article}
\usepackage{amsmath,amsthm,amssymb,amscd}
\usepackage{graphicx,color}

\title {The Group of Primitive Almost Pythagorean Triples}

\author{Nikolai A. Krylov and Lindsay M. Kulzer}

\date {}

\begin{document}

\newtheorem{thm}{Theorem}
\newtheorem{lem}{Lemma}
\newtheorem{claim}{Claim}
\newtheorem{dfn}{Definition}
\newtheorem{prop}{Proposition}

\def\Natu           {\mathbb N}
\def\Inte           {\mathbb Z}
\def\Rati           {\mathbb Q}
\def\Real           {\mathbb R}
\def\Field          {\mathbb F}
\def\P              {{\cal P}}
\def\F              {{\cal F}}
\def\A              {{\cal A}}

\def\lla            {\longleftarrow}
\def\lra            {\longrightarrow}
\def\ra             {\rightarrow}
\def\hra            {\hookrightarrow}
\def\lmt            {\longmapsto}
\def\lam            {\lambda}
\def\del            {\delta}
\def\eps            {\epsilon}

\maketitle

\parskip=3mm

\begin{abstract}
We consider the triples of integer numbers that are solutions of the
equation $x^2+qy^2=z^2$, where $q$ is a fixed, square-free arbitrary
positive integer. The set of equivalence classes of these triples
forms an abelian group under the operation coming from complex
multiplication. We investigate the algebraic structure of this group
and describe all generators for each $q\in\{2,3,5,6\}$. We also show
that if the group has a generator with the third coordinate being a
power of 2, such generator is unique up to multiplication by $\pm1$.
\end{abstract}

\noindent {\bf Keywords}: Pythagorean triples; Infinitely generated commutative groups\\
{\bf 2010 Mathematics Subject Classification}: 20K20, 11D09.

\section{Introduction and the group of PPTs}

The set of Pythagorean triples has various interesting structures.
One of such structures is induced by a binary operation introduced
by Taussky in \cite{Taussky}. Recall that a Pythagorean triple (PT
from now on) is an ordered triple $(a,b,c)$ of natural numbers
satisfying the identity $a^2+b^2 = c^2$, and given two such triples
$(a_1,b_1,c_1)$ and $(a_2,b_2,c_2)$ we can produce another one using
the following operation

\begin{equation}
\label{operation} A:= a_1a_2 + b_1b_2, ~~B:= |a_1b_2- a_2b_1|, ~~
C:= c_1c_2.
\end{equation}

The natural relation $(a,b,c)\simeq (\lambda a, \lambda b, \lambda
c)$ for $\forall \lambda\in\Natu$, called projectivization, is an
equivalence relation on this set. The operation mentioned above
induces an abelian group structure on the set of equivalence classes
of PTs where the identity element is the class of $(1,0,1)$. When
$a,b$ and $c$ have no common prime divisors, the triple $(a,b,c)$ is
called {\it primitive}. It's easy to see that every equivalence
class contains exactly one primitive Pythagorean triple. Thus the
set of all primitive Pythagorean triples (PPTs from now on) forms an
abelian group under the operation given in (\ref{operation}). The
algebraic structure of this group, denoted by {\bf P}, was
investigated by Eckert in \cite{Eckert}, where he proved that the
group of PPTs is a free abelian group generated by all primitive
triples $(a,b,c)$, where $a>b$ and $c$ is a prime number of the
linear form $c=4n+1$. Every Pythagorean triple $(a,b,c)$ naturally
gives a point on the unit circle with rational coordinates
$(a/c,b/c)$ and the equivalence class of PTs corresponds to a unique
point on the circle. Operation (\ref{operation}) on the Pythagorean
triples corresponds to the ``angle addition" of rational points on
$S^1$ and thus the group of PPTs is identified with the subgroup of
all rational points on $S^1$. Analysis of this group was done by Tan
in \cite{Tan} and his Theorem 1 (see page 167) is equivalent to what
Eckert proved in his Proposition on page 25 of \cite{Eckert}.

It is not hard to notice that the composition law (\ref{operation})
naturally extends to the solutions of the Diophantine equation
\begin{equation}
\label{main} X^2 + q\cdot Y^2 = Z^2
\end{equation}
where $q$ is a fixed, square-free arbitrary positive integer. Via
projectivization, we obtain a well defined binary operation on the
set of equivalence classes of solutions to (\ref{main}), and the set
of such classes forms an abelian group as well. For some special
values of $q$, including all $q\in\{2,3,5,6,7,15\}$, such a group
has been considered by Baldisserri (see \cite{Baldisserri}).
However, it seems that the generators $(3,1,4)$ for $q=7$, and
$(1,1,4)$ for $q=15$ are missing in \cite{Baldisserri}.

With the above in mind, we will consider in this paper the set of
triples we call {\it almost Pythagorean triples}, which are
solutions to the equation (\ref{main}). As in the case of PTs, each
equivalence class here contains exactly one {\it primitive} almost
Pythagorean triple and therefore the set of equivalence classes is
the set of {\it Primitive Almost Pythagorean Triples} (PAPTs from
now on).

In the next two sections we give a complete description of this
group for $q\in\{2,3,5,6\}$, similar to the one given in
\cite{Eckert}. We also prove that for all $q\neq 3$ the group of
PAPTs is free abelian of infinite rank. In the last section we will
discuss solutions $(a,b,c)$ where $c$ is even. Please note that some
of the results we prove here have been obtained earlier by
Baldisserri, however our proof of existence of elements of finite
order is different from the one given in \cite{Baldisserri}. We also
explain that if $(a,b,2^k)$ is a non-trivial solution of
(\ref{main}) with $q\neq 3$, the set of all such solutions makes an
infinite cyclic subgroup of the group of PAPTs. When $q=7$ and
$q=15$ such a subgroup is missing in the Theorem 2. of
\cite{Baldisserri}.

\section{Group of PAPTs}

Let  $T_q$ denote the set of all integer triples $(a, b, c)\in \Inte
\times\Inte \times \Natu$ such that $a^2 + q\cdot b^2 = c^2$. We
introduce the following relation on $T_q$: two triples $(a, b, c)$
and $(A, B, C)$ are equivalent if there exist $m, n \in \Inte
\setminus \{0\}$ such that $m(a, b, c) = n(A, B, C)$, where $m(a, b,
c) = (ma, mb, |mc|)$. It is a straight forward check that this is an
equivalence relation (also known as {\it projectivization}). We will
denote the equivalence class of $(a,b,c)$ by $[a,b,c]$. Note that
$[a,b,c] = [-a,-b,c]$, but $[a,b,c] \neq [-a,b,c]$.  We will denote
the set of these equivalence classes by $\P_q$. Now we define a
binary operation on $\P_q$ that generalizes the one on the set of
PPTs defined by (\ref{operation}).

\begin{dfn}
For two arbitrary classes $[a, b, c],~[A, B, C] \in \P_q$ define
their sum by the formula
$$
[a, b, c] + [A, B, C] := [aA-qbB, aB+bA, cC].
$$
\end{dfn}

It is a routine check that this definition is independent of a
particular choice of a triple and thus the binary operation is well
defined. Here are two examples:\\
If $ q=7,~[3, 1, 4] +[3,1,4]+[3,1,4] = [3,1,4] + [2,6,16] =[-36, 20,
64] = [-9,5, 16].$\\ If $q= 14$,  $[5,2,9] +[13,2,15] = [9,36,135] =
[1,4,15].$

\noindent Since $[a,b,c] +[1,0,1] = [a,b,c],~ [a,b,c] +[-a,b,c] =
[-a^2-qb^2, 0, c^2] = [c^2, 0, c^2]$,  and the operation is
associative (this check is left for the reader), we obtain the
following (c.f. \S2 of \cite{Baldisserri} or \S4.1 of
\cite{Weintraub})

\begin{thm}
$(\P_q,~+)$ is an abelian group. The identity element is $[1,0,1]$
and the inverse of $[a, b, c]$ is $[a, -b, c]=[-a,b,c]$.
\end{thm}

The purpose of this paper is to see what the algebraic structure of
$(\P_q,~+)$ is, and how it depends on $q$. From now on we will
denote this group simply by $\P_q$. Please note that every
equivalence class $[a,b,c] \in \P_q$ can be represented uniquely by
a primitive triple $(\alpha, \beta, \gamma ) \in T_q$, where $\alpha
> 0$. In particular, this gives us freedom to refer to primitive triples to
describe elements of the group.



\noindent \underline{Remark 1}: The group $\P_q$ is a natural
generalization of the group ${\bf P}$ of PPTs. However, $\P_1$ is
not isomorphic to ${\bf P}$. The key point here is that the triple
$(0,1,1) \notin T_q$, when $q>1$, and the inverse of $[a,b,c]$ is
$[a,-b,c]=[-a,b,c]$. In particular, it forces the consideration of
triples with $a$ and $b$ being all integers and not only positive
ones. As a result, the triples $(1,0,1)$ and $(0,1,1)$ are not
equivalent in $T_1$. In order for the binary operation on the set of
PPTs to be well defined, the triple $(0,1,1)$ must be equivalent to
the identity triple $(1,0,1)$ (see formulae (5) on page 23 of
\cite{Eckert}). The relation between our group $\P_1$ and the group
${\bf P}$ of PPTs is given by the following direct sum decomposition
$$
\P_1 \cong {\bf P}\oplus \Inte/2\Inte,
$$
where the 2-torsion subgroup $\Inte/2\Inte$ is generated by the
element $[0,1,1]$. To prove this, one uses the map $f:{\bf P}\oplus
\Inte/2\Inte \lra \P_1$ defined by the following formula.
$$
f\bigl((a,b,c),n\bigr): = \left\{
\begin{array}{lcc}
[a,b,c]+[1,0,1] = [a,b,c] & \mbox{if} & n=0\\

[a,b,c]+[0,1,1] = [-b,a,c] & \mbox{if} & n=1\\
\end{array} \right.
$$

\noindent It's easy to see that this $f$ is an isomorphism.

\noindent \underline{Remark 2}: The group $\P_q$ also has a
geometric interpretation: Consider the set $\P(\Rati)$ of all points
$(X,Y)\in \Rati\times \Rati$ that belong to the conic $X^2+qY^2=1$.
Let $N=(1,0)$ and take any two $A,B\in \P(\Rati)$. Draw the line
through $N$ parallel to the line $(AB)$, then its second point of
intersection with the conic $X^2+qY^2=1$ will be $A+B$ (see
\cite{Lemmermeyer3}, section 2.2 and also section 1 of
\cite{LemmermeyerC} for the details). Via such geometric point of
view, Lemmermeyer draws a close analogy between the groups
$\P(\Inte)$ of integral points on the conics in the affine plane and
the groups $E(\Rati)$ of rational points on elliptic curves in the
projective plane (\cite{Lemmermeyer3}, \cite{LemmermeyerC}). One of
the key characteristics of $\P(\Inte)$ and $E(\Rati)$ is that both
of the groups are finitely generated. Note however that if $q>0$,
the curve $X^2+qY^2=1$ has only two integer points $(\pm 1, 0)$. One
could consider the solutions of $X^2+qY^2=1$ over a finite field
$\Field_q$ or over the $p$-adic numbers $\Inte_p$. In each of these
cases the group of all solutions is also finitely generated and we
refer the reader to section 4.2 of \cite{LemmermeyerC} for the exact
formulas. In the present paper we investigate the group structure of
all rational points on the conic $X^2+qY^2=1$ when $q\geq 2$ and
such group is never finitely generated, as we explain below.

\section{Algebraic structure of $\P_q$}

The classical enumeration of primitive pythagorean triples in the
form
$$
(a,b,c) =
(u^2-v^2,~2uv,~u^2+v^2)~~~\mbox{or}~~~\left(\frac{u^2-v^2}{2},~uv,~\frac{u^2+v^2}{2}\right)
$$
is a useful component in understanding the group structure on the
set of PPTs. We assume here that integers $u$ and $v$ have no common
prime divisors, otherwise $(a,b,c)$ won't be primitive. One could
use the Diophantus chord method (see for example \S 1.7 of
\cite{Stillwell}) to derive such enumeration of all PPTs. This
method can be generalized to enumerate all solutions to (\ref{main})
for all square-free $q>1$. In particular, if a primitive triple
$(a,b,c)\in T_q$, then there exists a pair $(u,v)$ of integers with
no common prime divisors, such that
$$
(a,b,c) = (\pm(u^2-q v^2),~2uv,~u^2+q v^2)~~~\mbox{or}~~~
\left(\pm\frac{u^2-q v^2}{2},~uv,~\frac{u^2+q v^2}{2}\right).
$$

We can use this enumeration right away to prove that if $c$ is
prime, and $(a,b,c)\in T_q$, then such a pair of integers $(a,b)$ is
essentially unique. Here is the precise statement.

\begin{claim}
If $c$ is prime and
$$
x^2 + qy^2 = c^2 = a^2 + q b^2,~~\mbox{where}~~abxy\neq 0
$$
then $(x,~y) = (h_1a,~h_2b)$, where $h_i=\pm1$.
\end{claim}
\begin{proof}
We apply Lemma 5.48 from \S 5.5. of \cite{Weintraub}. When
$2c=u^2+qv^2$ the proof needs an additional argument explaining why
not just $\beta/\alpha_0$ but $\beta/(2\alpha_0)$ will be in the
ring of integers. It can be easily done considering separate cases
of even and odd $q$ and using the fact that if $q$ is odd, then $u$
and $v$ used in the enumeration are both odd, and if $q$ is even,
then $u$ will be even and $v$ will be odd. We leave details to the
reader.
\end{proof}

We will use these results when we discuss generators of $\P_q$
below, but first we will find for which $q>1$ the group $\P_q$ will
have elements of finite order.

\subsection{Torsion in $\P_q$}

We follow Eckert's geometric argument (\cite{Eckert}, page 24) to
understand the torsion of $\P_q$.

\begin{lem}
If $q=2$ or $q>3$, then $\P_q$ is torsionfree. $\P_3\cong\F_3 \oplus
\Inte/3\Inte$, where $\F_3$ is a free abelian group.
\end{lem}
\begin{proof}
Let us assume that $q\geq 2$, and suppose the triple $(a,b,c)$ is a
solution of (\ref{main}), that is we can identify point
$(a/c,\sqrt{q}\cdot b/c)$ with $e^{i\alpha}$ on the unit circle
${\bf U}$. Then a circle $S_r^1$ with radius $r=\alpha/(2\pi)$ is
made to roll inside ${\bf U}$ in the counterclockwise direction. The
radius $r$ is chosen this way so that the length of the circle
$S_r^1$ equals length of the smaller arc of ${\bf U}$ between the
points $e^{i\alpha}$ and $e^0=(1,0)$. Let us denote the point
$(1,0)$ by $P$ and assume that this point moves inside the unit disk
when $S_r^1$ rolls inside ${\bf U}$. When $1=kr$ for some positive
integer $k$, this point $P$ traces out a curve known as a
hypocycloid. In this case the point $P$ will mark off $k-1$ distinct
points on ${\bf U}$ and will return to its initial position $(1,0)$
so the hypocycloid will have exactly $k$ cusps. If $P$ doesn't
return to $(1,0)$ after the first revolution around the origin, it
might come back to $(1,0)$ after, say $n$, such revolutions. In that
case $n\cdot 2\pi = m\cdot \alpha$, for some $m\in\Natu$. Thus,
$\alpha$ is a rational multiple of $\pi$, or to be more precise,
$$
\alpha = \pi\cdot \frac{2n}{m}
$$
Due to Corollary 3.12 of \cite{Niven} (see Ch.3, Sec.5), in such a
case the only possible rational values of $\cos(\alpha)$ are
$0,\pm\frac{1}{2}, \pm1$. Since $\cos(\alpha) =a/c$, where $a\neq
0$, we see that $\P_q$ might have a torsion only if $a/c = \pm1/2$
or $a/c=\pm1$. In the latter case we must have $q\cdot b^2 = 0$,
which implies that the element $[a,b,c]$ is the identity of $\P_q$.
Suppose now $a/c = \pm1/2$. Then $qb^2=3a^2$ and if $3\neq q$ we
will have a prime $t\neq 3$ dividing $q$. We can assume without loss
of generality that $\gcd(a,b)=1$, hence we obtain $t|a$ and
therefore $t^2|qb^2$. Since $q$ is square-free, we must have
$t|b^2$, which contradicts that $\gcd(a,b)=1$. Therefore if $q=2$ or
$q>3$, $\P_q$ is torsionfree. Suppose now $q=3$. Then we obtain
$a=\pm b$ and we can multiply $[a,b,c]$ by $-1$, if needed, to
conclude that $[a,b,c]=[1,1,2]$ or $[a,b,c]=[1,-1,2]$. We have
$\langle [a,b,c]\rangle \cong \Inte/3\Inte$ in both these cases. It
implies that $\P_3/(\Inte/3\Inte)$ is free abelian and hence
$\P_3\cong\F_3 \oplus \Inte/3\Inte$.
\end{proof}

\noindent \underline{Remark 3}: There is a another way to obtain
this lemma via a different approach to the group $\P_q,~q > 0$. The
authors are very thankful to Wladyslaw Narkiewicz who explained this
alternative viewpoint to us (cf. also with \cite{Baldisserri}).
Consider an imaginary quadratic field $\Rati(\sqrt{-q})$ and the
multiplicative subgroup of non-zero elements whose norm is a square
of a rational number. Let us denote this subgroup by $\A_q$.
Obviously $\Rati^*\subset \A_q$ ($\Rati^*$ denotes the group of
non-zero rational numbers). It is easy to see that $\P_q\cong
\A_q/{\Rati^*}$, and it follows from Theorem A. of Schenkman (see
\cite{Schenkman}) that $\A_q$ is a direct product of cyclic groups.
Hence the same holds for $\P_q$. If $q=1$ or $q=3$ the group $\A_q$
will have elements of finite order since the field
$\Rati(\sqrt{-q})$ has units different from $\pm1$. These units will
generate in $\P_q$ the torsion factors $\Inte/2\Inte$ or
$\Inte/3\Inte$, when $q=1$ or $q=3$ respectively.

\subsection{On generators of $\P_q$ when $q\leq 6$}

In this subsection we assume that $2\leq q \leq 6$, and will
describe the generators of $\P_q$ similar to the way it was done by
Eckert in his proposition on pages 25 and 26 of \cite{Eckert}. We
will use $\F_q$ to denote the free subgroup of $\P_q$. As follows
from 3.1 above, $\F_q = \P_q$, for $q\neq 3$, and $\P_3\cong
\F_3\oplus(\Inte/3\Inte)$.

The key point in Eckert's description of the generators of the group
of primitive pythagorean triples is the fact that a prime $p$ can be
a hypothenuse in a pythagorean triangle if and only if $p\equiv 1
\pmod{4}$. Our next lemma generalizes this fact to the cases of
primitive triples from $T_q$, with $q\in\{2,3,5,6\}$.

\begin{lem}
If $(a,b,c) \in T_2$ is primitive and $p$ is a prime divisor of $c$,
then there exist $u, v \in \Inte$ such that $p = u^2 + 2v^2$. If
$(a,b,c) \in T_3$ is primitive and $p$ is a prime divisor of $c$,
then either $p = 2$ or there exist $u, v \in \Inte$ such that $p =
u^2 +3v^2$. If $(a,b,c) \in T_q$ is primitive where $q =5$ or $q=6$,
and $p$ is a prime divisor of $c$, then $\exists u, v \in \Inte$
such that $p = u^2 +qv^2$ or $2p= u^2 +qv^2$.
\end{lem}
\begin{proof}
Consider $(a,b,c) \in T_q$. Since $a^2 + qb^2 = c^2$ where $q \in \{
2, 3, 5, 6\}$, it follows from the generalized Diophantus chord
method that $\exists s, t \in \Inte$ such that $c = s^2 + qt^2$ or
$2c=s^2 + qt^2$. Suppose $c = p_1^{n_1} \cdot \ldots \cdot
p_k^{n_k}$, is the prime decomposition of $c$.

\noindent \underline{Case 1}: $q = 2$. We want to show that each
prime $p_i$ dividing $c$ can be written in the form $p_i = u^2
+2v^2$ for some $u, v \in \Inte$ (note that if $q$ is even, $p_i\neq
2$). It is well known that a prime $p$ can be written in the form
$$
p = u^2 +2v^2 \iff p = 8n +1 ~~\mbox{or}~~ p = 8n+3,~~ \mbox{for
some integer} ~n
$$ (see chapter 9 of \cite{Stillwell}, or chapter 1 of
\cite{Cox}). Thus it's enough to show that if a prime $p \vert c$
then $p = 8n +1$ or $p = 8n+3$. Since $p \vert c$, and $c = s^2 +
qt^2$ or $2c=s^2 + qt^2$ we see that $\exists m \in \Inte$ such that
$pm = s^2 +2t^2$ and hence $-2t^2 \equiv s^2\pmod{p}$, i.e. the
Legendre Symbol $(\frac{-2t^2}{p}) = 1$. Using basic properties of
the Legendre symbol, it implies that $(\frac{-2}{p})=1$. But
$(\frac{-2}{p})=1$ iff $p = 8n +1$ or $p = 8n+3$ as follows from the
supplements to quadratic reciprocity law. This finishes the case
with $q = 2$.

\noindent \underline{Case 2}: Suppose now that $q=3$. Then $(1,1,2)
\in T_3$ gives an example when $c$ is divisible by prime $p = 2$.
Note also that prime $p=2$ is of the form $2p = u^2 +3v^2$. Assuming
from now on that prime $p$ dividing $c$ is odd, we want to show that
there exist $u, v \in \Inte$ such that $p = u^2+3v^2$, which is true
if and only if $\exists n \in \Inte$ such that $p = 3n+1$ (see again
\cite{Stillwell} or \cite{Cox}). Hence, in our case, it suffices to
show that if $p\vert c$ then $\exists n \in \Inte$ such that $p =
3n+1$. As in Case 1, $\exists m \in \Inte$ such that $pm = s^2
+3t^2$ for some $s,t \in \Inte$. Therefore, we have that the
Legendre Symbol $(\frac{-3}{p}) = 1$, which holds iff $p = 3n+1$.
One can prove this using the quadratic reciprocity law (e.g.
\cite{Stillwell}, Section 6.8).

\noindent \underline{Case 3}: Suppose now that $q = 5$. Note that in
this case $c$ must be odd. Indeed, if $c$ was even, $x^2+5y^2$ would
be divisible by 4, but on the other hand, since both of $x$ and $y$
must be odd when $q$ is odd and $c$ is even, we see that
$x^2+5y^2\not\equiv 0\pmod{4}$. Since $p \vert c$ then again
$\exists m \in\Inte$ such that $pm = s^2 +5t^2$ for some $s,t \in
\Inte$. I.e. $(\frac{-5}{p}) = 1$. It is true that for any integer
$n$ and odd prime $p$ not dividing $n$ that Legendre Symbol
$(\frac{-n}{p}) = 1$ iff $p$ is represented by a primitive form
$ax^2 +bxy + cy^2$ of discriminant $-4n$ such that $a,b,$ and $c$
are relatively prime (see Corollary 2.6 of \cite{Cox}). Following an
algorithm in \S 2.A of \cite{Cox} to show that every primitive
quadratic form is equivalent to a reduced from, one can show that
the only two primitive reduced forms of discriminant $-4 \cdot 5 =
-20$ are $x^2+5y^2$ and $2x^2 +2xy + 3y^2$. Through a simple
calculation its easy to see that a prime $p$ is of the form
$$
p=2x^2 +2xy + 3y^2~~\iff~~ 2p = x^2+5y^2.
$$
This finishes the third case.

\noindent \underline{Case 4}: Lastly, let's consider the case when
$q = 6$. Once again since $p \neq 2$ and $p  \vert c$ then
$(\frac{-6}{p}) = 1$. Using the same Corollary used in case 3, we
see that $p$ must be represented by a primitive quadratic form of
discriminant $-4 \cdot 6 = -24$. Also, following the same algorithm
used in case 3 to determine such primitive reduced forms, we find
that there are only two; $x^2+6y^2$ and $2x^2 + 3y^2$. Through a
simple calculation it can be determined that a prime $p$ is of the
from
$$
p = 2x^2 + 3y^2 ~~\iff~~2p = x^2+ 6y^2.
$$
Thus, the lemma is proven.
\end{proof}

\noindent \underline{Remark 4}: One could write prime divisors from
this lemma in a linear form if needed. It is a famous problem of
classical number theory which primes can be expressed in the form
$x^2+ny^2$. The reader will find a complete solution of this problem
in the book \cite{Cox} by Cox. For example, if $p$ is prime, then
for some $n\in\Inte$ we have
$$
p = \left\{ \begin{matrix} 20n+1 \\ 20n+3 \\20n+7
\\20n+9 \end{matrix} \right.
$$
if and only if $p=x^2+5y^2$ or $p=2x^2+2xy+3y^2$. We refer the
reader for the details to chapter 1 of \cite{Cox}.

Now we are ready to describe all generators of $\P_q$, where
$q\in\{2,3,5,6\}$. Our proof is similar to the proof given in
\cite{Eckert} by Eckert, where he decomposes the hypothenuse of a
right triangle into the product of primes and after that peels off
one prime at a time, together with the corresponding sides of the
right triangle. His description of prime $p \equiv 1 \pmod{4}$ is
equivalent to the statement that $p$ can be written in the form
$p=u^2+v^2$, for some integers $u$ and $v$, which is the case of
Fermat's two square theorem. In the theorem below we also use
quadratic forms for the primes.

\begin{thm}
Let us fix $q \in \{2,3,5,6\}$. Then $\P_q$ is generated by the set
of all triples $(a,b,p) \in T_q$ where $a>0$, and $p$ is prime such
that $\exists u,v \in \Inte$ with $p=u^2 +qv^2$, or $2p=u^2 +qv^2$.
\end{thm}
\begin{proof}
Take arbitrary $[r, s, d] \in \P_q$ and let us assume that $(r,s,d)
\in T_q$ will be the corresponding primitive triple with $r>0$, and
the following prime decomposition of $d = p_1^{n_1} \cdot \ldots
\cdot p_k^{n_k}$. It is clear from what we've said above that $d$
will be odd when $[r,s,d] \in \F_q$, and $d$ will be even only if
$q=3$ and $[r,s,d] \notin \F_3$. Our goal is to show that
$$
[r,s,d]  = \sum \limits_{i=1}^{k} n_i\cdot[a_i, b_i,
p_i],~\mbox{where}~a_i > 0,~~n_i\cdot[a_i,b_i,p_i] :=
\underbrace{[a_i, b_i, p_i] + \cdots + [a_i, b_i, p_i]}_{\mbox{$n_i$
times}}
$$
and $p_i$ is either of the form $u^2 +qv^2$, or of the form
$(u^2+qv^2)/2$. We deduce from our Lemma 2 that each prime $p_i~|~d$
can be written in one of these two forms. Hence, for all
$p_i,~\exists a_i, b_i \in \Inte$ such that $ a_i^2 + qb_i^2 =
p_i^2$. Indeed, if we have $2p=u^2+qv^2$, then
$$
4p^2 = (u^2-qv^2)^2 + 4q(uv)^2
$$
and since $u^2+qv^2$ is even, $u^2-qv^2$ will be even as well, and
therefore we could write $\alpha^2 + q\beta^2 = p^2$, where
$\alpha=(u^2-qv^2)/2$ and $\beta=uv$. Thus $[a_i,b_i,p_i] \in \P_q$.

\noindent Since $\P_q$ is a group, the equations
$$
[r,s,d] = \left\{
\begin{array}{l}
[X_1, Y_1, D_1] + [a_k,b_k,p_k] \\

[X_2, Y_2, D_2] + [-a_k,b_k,p_k] \\
\end{array} \right.
$$

\noindent always have a solution with $(X_i,Y_i,D_i) \in \Inte
\times \Inte\times\Natu$. The key observation now is that only one
of the triples $(X_i,Y_i,D_i)$ will be equivalent to a primitive
triple $(x,y,d_1)$, with $d_1<d$. Indeed, we have $ [r,s,d] = [X, Y,
D] \pm [a,b,p] $ or
$$
[X,Y,D] = [r,s,d] \pm [-a,b,p] = \left\{
\begin{array}{lcc}
[-ra-qsb,rb-sa,dp] \\

[ra-qsb,rb+sa,dp] \\
\end{array}\right.
$$

\noindent Since $p~|~d$, we have $dp\equiv0\pmod{p^2}$ and hence it
is enough to show that either $ra+qsb\equiv rb-sa\equiv 0
\pmod{p^2}$, or $ra-qsb\equiv rb+sa\equiv 0 \pmod{p^2}$ (c.f. Lemma
on page 24 of \cite{Eckert}). From the following identity
$$
(sa-rb)(sa+rb) = s^2a^2-r^2b^2 = s^2(a^2+qb^2) -
b^2(r^2+qs^2)\equiv0\pmod{p^2},
$$
we deduce that either $p$ divides each of $sa-rb$ and $sa+rb$, or
$p^2$ divides exactly one of these two terms. In the first case
$p~|~2sa$, which is impossible if $p$ is odd, since then either
$a^2>p^2$ or $(r,s,d)$ won't be primitive. If we assume $p=2$, then
as we explained in Lemma 2., $q=3$ and therefore $(a,b,p) = (1,1,2)$
so $(ra-qsb,rb+sa,dp) = (r-3s,r+s,2d)$. But $r+s\equiv r-3s\pmod{4}$
and if $4~|~r+s$ we can write $(ra-qsb,rb+sa,2d) =
4\bigl((r-3s)/4,(r+s)/4,d_1\bigr)$, where $d_1=d/2$. If
$r+s\equiv2\pmod{4}$, we will divide each element of the other
triple by 4.

Thus we can assume from now on that $p$ is an odd prime and that
either $p^2~|~sa-rb$ or $p^2~|~sa+rb$. Let us assume without loss of
generality that $sa-rb = kp^2$ for some $k\in\Inte$. Since the
triple $(-ra-qsb,rb-sa,dp)$ is a solution of (\ref{main}), and the
last two elements are divisible by $p^2$, it is obvious that the
first element must be divisible by $p^2$ too, i.e. that $ra+qsb =
tp^2$. That implies that
$$
[X,Y,D] = [-ra-qsb,rb-sa,dp] = [-t, -k, d_1],
$$
where $d_1 = d/p<d$, which we wanted to show. The other case is
solved similarly. Note that only one of the two triples will have
all three elements divisible by 4, which means that only $[a,b,p]$
or $[-a,b,p]$ can be subtracted from the original element $[r,s,d]$
in such a way that the result will be in the required form.

Thus we can ``peel off" the triple $[a_k,b_k,p_k]$ from the original
one $[r,s,d]$ ending up with the element $[x,y,d_1]$, where new $d_1
< d$. Note that we can always assume that $a_k>0$ by using either
$[a_k,b_k,p_k]$ or $[-a_k,-b_k,p_k]$. Then simply keep ``peeling
off" until all prime divisors of $d$ give the required presentation
of the element $[r,s,d]$ as a linear combination of the generators
$[a_i,b_i,p_i]$.
\end{proof}

\noindent \underline{Remark 5}: Since these primes are the
generators of $\P_q$ when $q \in \{ 2, 3, 5, 6 \}$ and each prime
(with exception $p = 2$ when $q = 3$) generates an infinite cyclic
subgroup, it is obvious that $\P_q$ contains an infinite number of
elements.  The same holds for $\P_q$ when $q \geq 7$. This can be
shown through properties of Pell's equation $c^2 -qb^2 = 1$ where
$q$ is a square-free positive integer different from 1. This
equation can be re-written as $c^2 = 1^2 +qb^2$, which is in fact
our equation (\ref{main}) with specific solutions $(1, b, c)$.  It
is a classical fact of number theory that this equation always has a
nontrivial solution and in result, has infinitely many solutions
(see \cite{Weintraub}, Section 4.2 or \cite{Stillwell}, Section
5.9).

\noindent Note that it is not obvious that Pell's equation has a
nontrivial solution for arbitrary $q$. For example, the smallest
solution of the equation
$$
1^2 + 61b^2 = c^2~~\mbox{is}~~b = 226,153,980,~~c = 1,766,319,049.
$$
Let us observe that the equation $a^2+61b^2 = c^2$, where $a$ is
allowed to be any integer, has many solutions with ``smaller"
integer triples. Three examples are [3,16,125], [6,7,55], and
[10,9,71].

\subsection{On generators of $\P_q$ when $q\geq 7$ and the triples $(a,b,2^k)$}

It is interesting to see how the method of peeling off breaks down
in specific cases of $q$ for $q \geq 7$. Here are some examples of
PAPTs $(a,b,c) \in T_q$, where $c$ is divisible by a prime $p$ but
there exist no nontrivial pair $r,s \in \Inte$, such that $(r,s,p)
\in T_q$.

The primitive triple $(9,1,10) \in T_{19}$ is a solution, where 10
is divisible by primes 2 and 5, however, it is impossible to find
nonzero $a,~b \in \Inte$, such that $a^2 + 19b^2 = 5^2$.

The primitive triple $(3,1,4) \in T_7$ is a solution, where 4 is
divisible by prime 2, however, it is impossible to solve $a^2 + 7b^2
= 2^2$ in integers. In $T_{15}$ the primitive triple $(1,1,4)$ is a
problematic solution for the same reason.

It is mentioned in \cite{Baldisserri} (see Observation \#2 on page
304) that if a non-trivial and primitive $(a,b,c)$ solves
(\ref{main}), then $c$ could be even only when $q \equiv 3\pmod{4}$.
Moreover if $q \equiv 3\pmod{8}$, we must have $c=2\cdot\mbox{odd}$,
but if $q \equiv 7\pmod{8}$ we could have $c$ divisible by any power
of $2$. Indeed, as we just mentioned above, the triple $(3,1,4)$
solves (\ref{main}) with $q=7$, and clearly can not be presented as
a sum of two ``smaller" triples. Since $\P_7$ is free, we see that
$(3,1,4)$ must generate a copy of $\Inte$ inside $\P_7$, and one can
easily check that we have
$$
2\cdot[3,1,4] = \pm[1,3,2^3],~~~3\cdot[3,1,4] =\pm[9,5,2^4],~~~
4\cdot[3,1,4] =\pm[31,3,2^5],~~~ \ldots
$$
The same holds for the triple $(1,1,4) \in T_{15}$ but somehow these
two generators of $\P_7$ and $\P_{15}$ are not mentioned in the
theorem 2 of \cite{Baldisserri}.

Can we have more than one such generator for a fixed $q$? In other
words, how many nonintersecting $\Inte$-subgroups of $\P_q$ can
exist, provided that each subgroup is generated by a triple where
$c$ is a power of 2? The following theorem shows that there could be
only one such generator (for the definition of {\it irreducible
solution} we refer the reader to page 304 of \cite{Baldisserri}, but
basically it means that this solution is a generator of the group of
PAPTs).

\begin{thm}
Fix $q$ as above and assume that the triple $(a,b,2^k)$ is an
irreducible solution of (\ref{main}). If $(x,y,2^r)\in T_q$ and
$r\geq k$, then $\exists n ~ \in\Inte$ such that
$$
[x,y,2^r] =  n\cdot [a,b,2^k]
$$
\end{thm}

\begin{proof}
Our idea of the proof is to show that given such a triple
$(x,y,2^r)\in T_q$ with $r\geq k$, we can always ``peel of" (i.e.
add or subtract) one copy of $(a,b,2^k)$ so the resulting primitive
triple will have the third coordinate $\leq 2^{r-1}$. Thus we
consider
$$
[S,T,V] := [x,y,2^r] \pm [a,b,2^k] = \left\{
\begin{array}{lcc}
[xa-qyb,~ay + xb,~2^{r+k}] \\

[xa+qyb,~ay - xb,~2^{r+k}] \\
\end{array}\right.
$$

\noindent Since $a,~b,~x$ and $y$ are all odd, either $ay + xb$ or
$ay - xb$ must be divisible by 4. Let's assume that $4~|~ay-xb$ and
hence we can write $ay - xb = 2^d\cdot R$, where $d\geq 2$. Clearly,
it's enough to prove that $d\geq k+1$. We prove it by induction,
i.e. we will show that if $d\leq k$, then $R$ must be even.

\noindent Since $S = xa+qyb$ we could write
$$
\begin{pmatrix}
2^d\cdot R\\
S
\end{pmatrix} =
\begin{pmatrix}
-b & a\\
a & qb
\end{pmatrix} \cdot
\begin{pmatrix}
x\\
y
\end{pmatrix}
~~~\mbox{and hence}~~~
\begin{pmatrix}
x\\
y
\end{pmatrix} = \frac{1}{2^{2k}}\cdot
\begin{pmatrix}
qb & -a\\
-a & -b
\end{pmatrix} \cdot
\begin{pmatrix}
2^d\cdot R\\
S
\end{pmatrix}
$$
which gives $ bS = -2^{2k}y -a2^dR$. Since $(bS,bT,bV)\in T_q$, we
can also write
$$
(2^{2k}y + a2^dR)^2 + qb^2\cdot(2^dR)^2 = b^2\cdot2^{2r+2k}.
$$
This last identity is equivalent to the following one (after using
$a^2+qb^2 = 2^{2k}$ and dividing all terms by $2^{2k}$)
$$
2^{2k}y^2 + 2^{d+1}ayR + 2^{2d}R^2 = b^22^{2r}.
$$
Furthermore, we can cancel $2^{d+1}$ as well, because $1<d\leq k\leq
r$, and then we will obtain that
$$
ayR = b^22^{2r-d-1} - 2^{d-1}R^2 - 2^{2k-d-1}y^2 = \mbox{even},
$$
which finishes the proof since $a$ and $y$ are odd.

\end{proof}

\noindent \underline{Remark 6}: Please note that if a primitive
triple $(a,b,2\cdot d)\in T_{q}$ for $q \equiv 7\pmod{8}$, it is
easy to show that $d$ must be even (compare with Observation \# 2 of
\cite{Baldisserri}, where $\lambda$ must be at least 2). When
$q\in\{7,~15\}$, we obtain the generators $(3,1,4)$ and $(1,1,4)$
respectively. However, if for example $q=23$, the primitive solution
$(a,b,c)$ where $c$ is the smallest power of 2 is $(7,3,16)$ but
$(11,1,12)$ also belongs to $\P_{23}$.

\noindent Siena College, Department of Mathematics\\
515 Loudon Road, Loudonville NY 12211\\ ~ \\
{\small nkrylov@siena.edu {\it and} lm15kulz@siena.edu}

\end{document}